\documentclass[twoside]{amsart}

\usepackage[backref=page]{hyperref}
\usepackage{amsmath}
\usepackage{amsthm}
\usepackage{amssymb}
\usepackage[all]{xy}
\usepackage{hyperref}
\usepackage{enumerate}
\usepackage{color}
\usepackage{multirow}
\usepackage{mathrsfs} 
\usepackage{bm}       

\newtheorem{theorem}{Theorem}[section]
\newtheorem*{theorem1}{Theorem \ref{thm:main1}}
\newtheorem*{theorem2}{Theorem \ref{thm:main2}}

\newtheorem{proposition}[theorem]{Proposition}
\newtheorem{lemma}[theorem]{Lemma}
\newtheorem{corollary}[theorem]{Corollary}

\numberwithin{equation}{section}

\theoremstyle{definition}
\newtheorem{definition}[theorem]{Definition}
\newtheorem{example}[theorem]{Example}
\newtheorem{question}[theorem]{Question}
\newtheorem{bquestion}[theorem]{Bondal's Question}

\theoremstyle{remark}
\newtheorem{remark}[theorem]{Remark}

\newcommand{\Db}{{\rm D}^{\rm b}}

\newcommand{\Pic}{{\rm Pic}}

\newcommand{\rk}{{\rm rk}}

\newcommand{\kEnd}{{\cal{E}nd}}
\newcommand{\Hom}{{\rm Hom}}

\renewcommand{\dim}{{\rm dim}\,}

\newcommand{\linedef}[1]{\emph{#1}}

\newcommand{\category}[1]{\mathsf{#1}}
\newcommand{\Coh}{\category{Coh}}

\newcommand{\coker}{{\rm coker}}

\newcommand{\dual}{^{\vee}}

\newcommand{\cat}[1]{\begin{bf}#1\end{bf}}

\newcommand{\tensor}{\otimes}
\newcommand{\isom}{\cong}

\newcommand{\cali}{\mathcal}
\newcommand{\cal}{\mathscr}
\newcommand{\ka}{{\cal A}}
\newcommand{\kb}{{\cal B}}

\newcommand{\ke}{{\cal E}}
\newcommand{\kf}{{\cal F}}

\newcommand{\kk}{{\cal K}}

\newcommand{\kl}{{\cal L}}

\newcommand{\ko}{{\cal O}}

\newcommand{\kq}{{\cal Q}}
\newcommand{\kr}{{\cal R}}
\newcommand{\ks}{{\cal S}}
\newcommand{\kt}{{\cal T}}
\newcommand{\ku}{{\cal U}}
\newcommand{\kv}{{\cal V}}
\newcommand{\kx}{{\cal X}}
\newcommand{\ky}{{\cal Y}}
\newcommand{\kw}{{\cal W}}
\newcommand{\kz}{{\cal Z}}

\newcommand{\VV}{\mathbb{V}}
\newcommand{\ZZ}{\mathbb{Z}}

\newcommand{\PP}{\mathbb{P}}
\newcommand{\GG}{\mathbb{G}}
\newcommand{\GGG}{\mathbb{G}}

\newcommand{\zr}{\kx^r_{m,n}}
\newcommand{\Zr}{\kz^r_{m,n}}
\newcommand{\yr}{\ky^r_{m,n}}

\begin{document}

\title[HPD for determinantal varieties]{Homological projective duality for determinantal varieties}

\author{Marcello Bernardara}
\address{Institut de Math\'ematiques de Toulouse \\ 
Universit\'e Paul Sabatier \\ 
118 route de Narbonne \\
31062 Toulouse Cedex 9\\ 
France}
\email{marcello.bernardara@math.univ-toulouse.fr}

\author{Michele Bolognesi}
\address{Institut Montpellierain Alexander Grothendieck \\ 
	Universit\'e de Montpellier \\ 
	Case Courrier 051 - Place Eug\`ene Bataillon \\ 
	34095 Montpellier Cedex 5 \\ 
	France}
\email{michele.bolognesi@umontpellier.fr}

\author{Daniele Faenzi}
\address{Universit\'e de Bourgogne\\
Institut de Math\'ematiques de Bourgogne\\
UMR CNRS 5584\\
UFR Sciences et Techniques -- B\^atiment Mirande -- Bureau 310\\
9 Avenue Alain Savary \\ 
BP 47870 21078 Dijon Cedex \\ 
France}
\email{daniele.faenzi@u-bourgogne.fr}

\thanks{D. F. partially supported by GEOLMI ANR-11-BS03-0011}

\maketitle

\begin{abstract}
In this paper we prove Homological Projective Duality for categorical resolutions of several classes of linear determinantal varieties. By this we mean varieties that are cut out by the minors
of a given rank of a $m\times n$ matrix of linear forms on a given
projective space. As applications, we obtain pairs of
derived-equivalent Calabi-Yau manifolds, and address a question by A. Bondal asking whether the derived category of any smooth projective variety can be fully faithfully embedded in the derived category of a smooth Fano variety. Moreover we discuss the relation between rationality and categorical representability in codimension two for determinantal varieties.
\end{abstract}

\section{Introduction}

Homological Projective Duality (HPD) is one of the most exciting
recent breakthroughs in homological algebra and algebraic geometry. It
was introduced by A. Kuznetsov
in \cite{kuznetsov:hpd} and its goal is to generalize classical projective duality to a homological framework. One of the important features of HPD is that it offers
a very important tool to study the bounded derived category of
a projective variety together with its linear sections, providing interesting semiorthogonal decompositions as well as derived equivalences,
cf. \cite{kuznetsov:v14,kuz:4fold,kuznetsov:quadrics,auel-berna-bolo,kuz:icm}.

\medskip 

Roughly speaking, two (smooth) varieties $X$ and $Y$ are HP-dual if 
$X$ has an ample line bundle $\ko_X(1)$ giving a map $X \to \PP W$, $Y$ has an ample line bundle $\ko_Y(1)$ giving a map $Y \to \PP W\dual$,
and $X$ and $Y$ have dual semiorthogonal decompositions (called \it Lefschetz \rm decompositions) compatible with the projective embedding. In this case, given a 
generic linear subspace $L \subset W$ and its orthogonal $L^\perp \subset W\dual$, one can consider the linear  sections $X_L$ and $Y_L$ of $X$ and $Y$ respectively. 
Kuznetsov shows the existence of a category $\cat{C}_L$ which is
admissible both in $\Db(X_L)$ and in $\Db(Y_L)$, and whose orthogonal complement
is given by some of the components of the Lefschetz decompositions of $\Db(X)$ and $\Db(Y)$ respectively. That is, both $\Db(X_L)$ and $\Db(Y_L)$ admit
a semiorthogonal decomposition by a ``Lefschetz'' component, obtained
via iterated hyperplane sections, and a common ``nontrivial'' or ``primitive'' component.

HPD is closely related to classical projective duality: \cite[Theorem 7.9]{kuznetsov:hpd} states that the critical locus of the
map $Y \to \PP W\dual$ coincides with the classical projective dual of $X$. The main technical issue of this fact is that one has to take
into account singular varieties, since the projective dual of a smooth variety is seldom smooth - \it e.g. \rm the dual of certain Grassmannians are singular Pfaffian varieties \cite{boris-cald:pfaff}. On the other hand, derived (dg-enhanced)
categories should provide a so-called \it categorical \rm or \it non-commutative \rm resolution of singularities (\cite{kuznet:singul,vdb:non-comm-resos}).
Roughly speaking, one needs to find a sheaf of $\ko_Y$-(dg)-algebras $\kr$ such that the category $\Db(Y,\kr)$ of bounded complexes of coherent $\kr$-modules
is proper, smooth and $\kr$ is locally Morita-equivalent to some matrix algebra over $\ko_Y$ (this latter condition translates the fact that
the resolution is birational).
In the case where $Y$ is singular, one of the most difficult
tasks in proving HPD is to provide such a resolution with the required Lefschetz decomposition (for example, see \cite[\S 4.7]{kuz:icm}).
On the other hand, given a non-smooth variety, it is a very interesting question to provide such resolutions and study their properties
such as crepancy, minimality and so forth.

\medskip

The main application of HPD is that it is a direct method to produce semiorthogonal decompositions 
for projective varieties with non-trivial canonical sheaf, and derived equivalences for Calabi-Yau varieties. The importance of this application is due to the fact 
that determining whether a given variety admits or not a semiorthogonal decomposition is a very hard problem in general. Notice that there are cases where it is 
known that the answer to this question is 
negative, for example if $X$ has trivial
canonical bundle \cite[Ex. 3.2]{bridg-equiv-and-FM}, or if $X$ is a curve of positive genus \cite{okawa-curves}. On the other hand, if $X$ is Fano, then any line bundle is exceptional and gives then 
a semiorthogonal decomposition. Almost all the known cases of semiorthogonal decompositions of Fano varieties described in the literature
(see, \it e.g.\rm, \cite{kuznetsov:v14,kuz:fano,kuz:4fold,bolo_berna:conic,auel-berna-bolo}) can be obtained via HPD or its relative version.

\medskip

Derived equivalences of Calabi-Yau (CY for short) varieties have deep geometrical insight. First of all, it was shown by Bridgeland that birational CY-threefolds
are derived equivalent \cite{bridg-flop}. The converse in not true: the first example - that has been shown to be also a consequence of HPD in  \cite{kuznet-grass} - was
displayed by Borisov and Caldararu in \cite{boris-cald:pfaff}.

Besides their geometric relevance, derived equivalences between CY varieties play an important role in theoretical physics. 
First of all, Kontsevich's homological mirror symmetry 
conjectures an equivalence between the bounded derived category of a CY-threefold $X$ and the Fukaya category of its mirror. More recently, 
it has been conjectured that homological projective duality should be realized physically as phases of abelian \it gauged linear sigma models \rm (GLSM)
(see \cite{herbst-hori-page:phases}
and \cite{viennacircle:git}).

As an example, denote by $X$ and $Y$ the pair of equivalent CY--threefolds considered by Borisov and Caldararu.
R\o dland \cite{rodlando} argued that the families of $X$'s and $Y$'s (letting the linear section move in the ambient space) seem to have the same mirror 
variety $Z$ (a more string theoretical argument has been given recently by Hori and Tong \cite{hori-tong}).
The equivalence between $X$ and $Y$ would then fit Kontsevich's Homological
Mirror Symmetry conjecture via the Fukaya category of $Z$. It is thus fair to say that HPD plays an important role in understanding these questions and potentially
providing new examples. Notice in particular that some determinantal cases were considered in \cite{vivalafisica}.

\medskip

In this paper, we describe new families of HP Dual varieties. We consider two vector spaces $U$ and $V$ of dimension
$m$ and $n$ respectively with $m \leq n$. Let $\GGG=\GG(U,r)$ denote
the Grassmannian of $r$-dimensional quotients of $U$, set $\kq$ and $\ku$
for the universal quotient and sub-bundle respectively. Let $\kx := \PP(V
\tensor \kq)$ and $\ky := \PP(V \dual \tensor \ku^\vee)$,
for any $0 < r < m$.
Let $p:\kx \to \GGG$ and $q:\ky\to \GGG$ be
the natural projections. Set $H_X$ and $H_Y$ for the relatively ample
tautological divisors on $\kx$ and $\ky$.
Orlov's result \cite{orlov:proj_bundles}
provides semiorthogonal decompositions
\begin{equation}\label{eq:the-dual-lefsch-decos}
\begin{array}{l}
\Db(\kx) = \langle p^*\Db(\GGG), \ldots, p^*\Db(\GGG) \otimes \ko_\kx((r n-1)H_X) \rangle, \\
\\
\Db(\ky) = \langle q^*\Db(\GGG)  \otimes \ko_\ky(((r-m)n+1)H_Y), \ldots, q^*\Db(\GGG) \rangle.
\end{array}
\end{equation}

\begin{theorem1}
In the previous notation, $\kx$ and $\ky$ with Lefschetz decompositions \eqref{eq:the-dual-lefsch-decos} are HP-dual.
\end{theorem1}

The proof of the previous result is a consequence of Kuznetsov's HPD for projective bundles generated by global sections (see \cite[\S 8]{kuznetsov:hpd}). 
Here, the spaces of global sections of $\ko_\kx(H_X)$ and $\ko_\ky(H_Y)$
sheaves are, respectively, $W=V \otimes U$ and $W\dual=V\dual \otimes U\dual$.

The main interest of Theorem \ref{thm:main1} is that $\kx$ is known to
be the resolution of the variety $\kz^r$ of $m\times n$ matrices of rank at most $r$. Write
such a matrix as $M: U \to V\dual$.
Then $\kz^r$ is naturally a subvariety of $\PP W$, which is singular
in general, with resolution $f: \kx \to \kz^r$. Dually, $g:\ky \to
\kz^{m-r}$ is a desingularization of the variety of $m\times n$
matrices of corank at least $r$. 
Theorem \ref{thm:main1} provides the categorical framework to describe HPD between the classical projectively dual varieties
$\kz^r$ and $\kz^{m-r}$ (see, \it e.g.\rm, \cite{weyman-book}).

In the affine case, categorical resolutions for determinantal
varieties have been constructed by Buchweitz, Leuschke and van den Bergh \cite{buch-leu-vdbergh1,buch-leu-vdbergh2}. Such resolution is
crepant if $m=n$ (that is, in the case where $\kz^r$ has Gorenstein singularities). The starting point is Kapranov's construction of a full strong
exceptional collection on Grassmannians \cite{kapra-grassa}. One can use the decompositions  in exceptional objects \eqref{eq:the-dual-lefsch-decos} to produce
a sheaf of algebras $\kr'$ and a categorical
resolution of singularities $\Db(\kz^r,\kr') \simeq \Db(\kx)$.
For simplicity, we will denote by $\kr'$ the algebra on any of the
determinantal varieties $\kz^r$ (forgetting
about the dependence of $\kr'$ on the rank $r$). 
This gives a geometrically deeper version of Theorem \ref{thm:main1}.

\begin{theorem2}
In the previous notations, $\kz^r$ admits a categorical resolution of singularities $\Db(\kz^r,\kr')$, which is crepant
if $m=n$. Moreover,
$\Db(\kz^r,\kr')$ and $\Db(\kz^{m-r},\kr')$ are HP-dual.
\end{theorem2}

Once the equivalence $\Db(\kz^r,\kr') \simeq \Db(\kx)$ constructed, Theorem \ref{thm:main2}
is proved by applying directly Theorem \ref{thm:main1}. However, the geometric relevance of Theorem
\ref{thm:main2}, and its difference with Theorem \ref{thm:main1}, is that it shows HPD
directly on noncommutative structures over the determinantal varieties $\kz^r$ and $\kz^{m-r}$ with respect to their
natural embedding in $\PP W$ and $\PP W ^\vee$ respectively. That is, these natural
smooth and proper noncommutative scheme structures are well-behaved with
respect to projective duality and hyperplane sections. Finally, notice that whenever we pick
a smooth linear section $Z_L$ of $\kz^r$ (or a smooth section $Z^L$ of $\kz^{m-r}$), the restriction to $Z_L$ of the sheaf $\kr'$
is Morita-equivalent to $\ko_{Z_L}$, so that we get the derived category $\Db(Z_L)$ of the section itself.

As a consequence, given a matrix of linear forms on some projective space,
one can see the locus $Z$  where the matrix
has rank at most $r$ as a
linear section of $\kz^r$.
Assuming  $Z$ to have expected dimension, Theorem \ref{thm:main2} gives a categorical
resolution of singularities $\Db(Z,\kr')$ of $Z$ and
a semiorthogonal decomposition of this category involving the dual linear section of $\kz^{m-r}$.

Our construction of Homological Projective Duality allows us to recover some Calabi-Yau equivalences appeared in \cite{vivalafisica} and many more 
(see Corollary \ref{cor:functors}).

\medskip 

A special case is obtained by setting $r=1$. In this case $\kx$ is a
Segre variety and $\ky$ is the variety of degenerate matrices rank.

\medskip 

As an application of this new instance of Homological Projective
Duality, we try to address a fascinating question, asked by A. Bondal in Tokyo in 2011. 
Since any Fano variety admits semiorthogonal decompositions, it is natural to ask whether the derived category of any variety can be realized as a component 
of a semiorthogonal decomposition of a Fano variety. Under this perspective, considering Fano varieties will be enough to
study all ``geometric'' triangulated categories.

\begin{bquestion}\label{qu:bondal}
Let $X$ be a smooth and projective variety. Is there any smooth Fano variety $Y$ together with a full and faithful functor $\Db(X) \to \Db(Y)$?
\end{bquestion}

We will say that $X$ is \it Fano-visitor \rm if Question \ref{qu:bondal}
has a positive answer (see Definition \ref{host}).

\medskip 

On the other hand, an interesting geometrical insight of semiorthogonal decompositions is to provide a conjectural obstruction to rationality
of a given variety $X$. 
In \cite{bolognesi_bernardara:representability}, the first and second named authors introduced , based on existence of semiorthogonal decompositions, the notion of \it categorical representability \rm of a 
variety $X$ (see Definition \ref{def-cat-rep}). This notion allows to formulate a natural question about categorical obstructions
to rationality.

\begin{question}\label{qu:catrep=ration}
Is a rational projective variety always categorically representable in codimension at least 2? 
\end{question}

The motivating ideas of question \ref{qu:catrep=ration} can be traced back
to the work of Bondal and Orlov, and to their address at the 2002 ICM \cite{bondal_orlov:ICM2002}, and to Kuznetsov's
remarkable contributions (e.g. \cite{kuz:4fold} or \cite{kuz:rationality-report}).
Notice that a projective space is representable
in dimension 0. Roughly speaking, the idea supporting Question \ref{qu:catrep=ration}
is based on a motivic argument which let us suppose that birational transformations should not add components representable
codimension 1 or less (see also \cite{bolognesi_bernardara:representability}).

Several examples seem to suggest that Question \ref{qu:catrep=ration} may have a positive answer.
Let us mention conic bundles over minimal surfaces \cite{bolo_berna:conic}, fibrations in intersections of quadrics \cite{auel-berna-bolo}, 
or some classes of cubic fourfolds \cite{kuz:4fold}. Moreover, Question \ref{qu:catrep=ration}
is equivalent to one implication of Kuznetsov Conjecture on the rationality of a cubic fourfold \cite{kuz:4fold}, which was proved to
coincide with Hodge theoretical expectations for a general cubic fourfold by Addington and Thomas \cite{add-thomas}.

As consequences of Theorems \ref{thm:main1} and \ref{thm:main2}, we can show that (the categorical resolution of singularities of)
any determinantal hypersurface of general type is Fano visitor (\S \ref{sect:segre-det}), and that (the categorical resolution of singularities of)
a rational determinantal variety is categorically representable in codimension at least two (\S \ref{sect:fano}).
Hence we provide a large family of varieties for which Questions \ref{qu:bondal} and \ref{qu:catrep=ration} have positive answer.
As an example, we easily get the following corollary (compare with Example \ref{ex:plane-cves}).

\begin{corollary}
A smooth plane curve is Fano visitor.
\end{corollary}

\subsection*{Acknowledgments} We acknowledge A. Bondal for asking Question \ref{qu:bondal} at the conference ``Derived Categories 2011'' in Tokyo, which
has been a source of inspiration for this work. We thank J. Rennemo for pointing out a mistake in the first version
of this paper, and the anonymous referee for useful suggestions and questioning.
We are grateful to A. Kuznetsov, N. Addington and E. Segal for useful advises and exchange of ideas.

\section{Preliminaries}

\subsection{Notation}
We work over an algebraically closed field of characteristic zero $k$.
A vector space will be denoted by  a capital letter $W$; 
the dual vector space is denoted by $W^\vee$. Suppose $\dim (W)=N$, then the projective space of $W$ is denoted by $\PP W$ or simply by $\PP^{N-1}$.
We follow Grothendieck's convention, so that $\PP W$ is the set of
hyperplanes through the origin of $W$. The dual projective space is denoted by $\PP W ^\vee$ or by
$(\PP^{N-1})^\vee$.

\medskip

We assume the reader to be familiar with the theory of semiorthogonal 
decompositions and exceptional objects (see
\cite{bondal_orlov:semiorthogonal,huybrechts:libro,kuz:icm}).
Recall that
the summands of a semiorthogonal decomposition of a triangulated
category $\mathbf{T}$, by definition
are full triangulated subcategories of $\mathbf{T}$ which are
admissible, i. e. such that the inclusion admits a left and right adjoint.

\subsection{Categorical resolutions of singularities}
By a \linedef{noncommutative scheme} we mean (following Kuznetsov
\cite[\S2.1]{kuznetsov:quadrics}) a scheme $X$ together with a
coherent $\ko_{X}$-algebra $\ka$.  Morphisms are defined accordingly.
By definition, a noncommutative scheme $(X,\ka)$ has $\Coh(X, \ka)$, the category of coherent $\ka$-modules, as
category of coherent sheaves and $\Db(X,\ka)$ as bounded derived
category.

Following Bondal--Orlov
\cite[\S5]{bondal_orlov:ICM2002}, a \it categorical \rm(or \it noncommutative\rm) \it resolution of
singularities \rm $(X,\ka)$ of a possibly singular proper scheme $X$ is a torsion
free $\ko_X$-algebra $\ka$ of finite rank such that $\Coh(X,\ka)$ has
finite homological dimension (\it i.e.\rm, is smooth in the noncommutative
sense).

\begin{definition}
Let $X$ be a scheme. An object $T$ of $\Db(X)$ is called a {\em compact generator} if
$T$ is perfect and, for any object $S$ of $\Db(X)$, we have that the fact that $\Hom_{\Db(X)}(S,T[i])=0$ for all integers
$i$ is equivalent to $S=0$. Notice that, if $X$ is smooth and proper, the natural inclusion $\mathrm{Perf}(X)
\subset \Db(X)$ of perfect complexes into $\Db(X)$ is an equivalence. Hence any object in
$\Db(X)$ is perfect.
\end{definition}

In the case where $X$ admits a full exceptional collection, there is an explicit
compact generator $T$.

\begin{proposition}[\cite{bondal_kapranov:reconstructions}]
Suppose that $X$ is smooth and proper, and that $\Db(X)$ is generated by a full exceptional sequence
$\Db(X) = \langle E_1, \ldots, E_s \rangle$. Then $E= \oplus_{i=1}^s E_i$ is a compact generator.
In particular, consider the dg-$k$-algebra $\mathrm{End}(E)$. Then there is an equivalence of triangulated categories $\Db(X) \simeq \Db(\mathrm{End}(E))$.
\end{proposition}

\subsection{Homological Projective Duality}

Homological Projective Duality (HPD) was introduced by Kuznetsov
\cite{kuznetsov:hpd} in order to study derived categories of
hyperplane sections (see also \cite{kuznetsov:hyp-sections}).

Let us first recall the basic notion of HPD
from \cite{kuznetsov:hpd}.  Let $X$ be a projective scheme
together with a base-point-free line bundle $\ko_X(H)$.

\begin{definition}
A \linedef{Lefschetz decomposition} of $\Db(X)$ with respect to
$\ko_X(H)$ is a semiorthogonal decomposition
\begin{equation}\label{eq:def-of-lefschetz}
\Db(X) = \langle \cat{A}_0, \cat{A}_1(H), \dotsc, \cat{A}_{i-1}((i-1)H) \rangle,
\end{equation}
with
$$
0 \subset \cat{A}_{i-1} \subset \dotsc \subset \cat{A}_0,
$$
Such a decomposition is said to be \it rectangular \rm if $\cat{A}_0= \ldots = \cat{A}_{i-1}$.
\end{definition}

Let $W := H^0(X,\ko_X(H))$, and
$f: X \to \PP W$ the map given by the linear system
associated with $\ko_X(H)$, so that $f^*
\ko_{\PP W}(1) \isom \ko_X(H)$. We denote by ${\cali X} \subset X \times
\PP W\dual$ the universal hyperplane section of $X$
\[\cali X:= \{(x,H)\in X\times \PP W^\vee | x \in H\}.\]

\begin{definition}
Let $f: X \to \PP W $ be a smooth projective scheme with a
base-point-free line bundle $\ko_X(H)$ and a Lefschetz decomposition as above.
A scheme $Y$ with a map $g: Y \to \PP W^\vee$ is called \linedef{homologically projectively dual} (or
the \linedef{HP-dual}) to $f: X \to \PP W$ with respect to the
Lefschetz decomposition \eqref{eq:def-of-lefschetz}, if there exists
a fully faithful functor $\Phi: \Db(Y) \to \Db(\cali X)$ giving the semiorthogonal decomposition
$$
\Db({\cali X}) = \langle \Phi (\Db(Y), \cat{A}_1(1) \boxtimes \Db(\PP W\dual), \dotsc,
\cat{A}_{i-1}(i-1) \boxtimes \Db(\PP W\dual) \rangle.
$$

\end{definition}

Let $N = \dim (W)$ and let $c\le N$ be an integer. Given a
$c$-codimensional linear subspace $L \subset W$, we define the linear subspace $\PP_L \subset \PP W$ of
codimension $c$ as $\PP(W/L)$.
Dually, we have a linear subspace $\PP^L=\PP L^\perp$ of dimension $c-1$
in $\PP W\dual$, whose defining equations are the elements of $L^\perp
\subset W\dual$.
We define the varieties:
\[
X_L = X \times_{\PP W} \PP_L, \qquad Y_L = Y \times_{\PP W^\vee} \PP^L.
\]

\begin{theorem}[{\cite[Theorem~1.1]{kuznetsov:hpd}}]
\label{thm:HPD}
Let $X$ be a smooth projective variety with a map $f: X \to \PP W$, and a Lefschetz decomposition
with respect to $\ko_X(H)$.  If $Y$ is HP-dual to
$X$, then:
\begin{itemize}
\item[(i)] $Y$ is smooth projective and admits a \linedef{dual}
Lefschetz decomposition
\[
\Db(Y) = \langle \cat{B}_{j-1}(1-j), \dotsc, \cat{B}_1 (-1),
\cat{B}_0 \rangle, \quad\quad \cat{B}_{j-1} \subset \dotsc
\subset \cat{B}_1 \subset \cat{B}_0
\]
with respect to the line bundle $\ko_Y(H)=g^*\ko_{\PP W^\vee}(1)$.
\item[(ii)] if $L$ is admissible, i.e. if
\[
\dim X_L = \dim X - c, \quad \mbox{and}\quad \dim Y_L = \dim Y + c - N,
\]
then there exist a triangulated category $\cat{C}_L$ and
semiorthogonal decompositions:
\begin{align*}
\Db(X_L) & = \langle \cat{C}_L, \cat{A}_c(1), \dotsc, \cat{A}_{i-1}(i-c) \rangle,\\
\Db(Y_L) & = \langle \cat{B}_{j-1} (N-c-j), \dotsc, \cat{B}_{N-c}(-1), \cat{C}_L \rangle.
\end{align*}
\end{itemize}
\end{theorem}

\begin{remark}
In general, HPD involves non smooth varieties. Indeed, as shown by Kuznetsov \cite[Theorem 7.9]{kuznetsov:hpd}
the critical locus of the map $g: Y \to \PP W^\vee$ is the classical projective dual $X^\vee$ of $X$, which is
rarely smooth even if $X$ is smooth. If $X$ (resp. $Y$) is singular, then we have to replace $\Db(X)$ (resp. $\Db(Y)$)
by a categorical resolution of singularities $\Db(X,\ka)$ (resp. $\Db(Y,\kb)$) in all the statements and definitions
of this section. Theorem \ref{thm:HPD} holds in this more general framework, where we have to consider $\Db(X_L,\ka_L)$
(resp. $\Db(Y_L,\kb_L)$) for $\ka_L$ (resp. $\kb_L$) the restriction of $\ka$ to $X_L$ (resp. of $\kb$ to $Y_L$) in item
(ii).
\end{remark}

\subsection{Categorical representability and Fano visitors}

First, let us recall the definition of categorical representability for a variety.
\begin{definition}[\cite{bolognesi_bernardara:representability}]
A triangulated category $\cat{T}$ is \linedef{representable in
dimension $j$} if it admits a semiorthogonal decomposition
$$\cat{T} = \langle \cat{A}_1, \ldots, \cat{A}_l \rangle,$$
and for all $i=1,\ldots,l$ there exists a smooth projective connected variety $Y_i$ with $\dim Y_i \leq j$, such that $\cat{A}_i$
is equivalent to an admissible subcategory of $\Db(Y_i)$.
\end{definition}

\begin{definition}[\cite{bolognesi_bernardara:representability}]
\label{def-cat-rep}
Let $X$ be a projective variety. We say that $X$ is {\em categorically representable}
in dimension $j$ (or equivalently in codimension $\dim(X)-j$) if there exists a categorical resolution of singularities
of $\Db(X)$ representable in dimension $j$.
\end{definition}

Based on Bondal's Question \ref{qu:bondal}, we introduce the following definition.

\begin{definition}\label{host}
A triangulated category $\cat{T}$ is \it Fano-visitor \rm if there exists a smooth Fano variety $F$ and a fully faithful functor $\cat{T} \to \Db(F)$ such that $\Db(F) = \langle \cat{T},\cat{T}^\perp \rangle$.
A smooth projective variety $X$ is said to be a \it Fano-visitor \rm if its derived category $\Db(X)$ is Fano-visitor.
\end{definition}
We remark that, having a fully faithful functor $\Db(X) \to \Db(F)$ is enough to have the required semiorthogonal decomposition
\cite{bondal:representations}.
Relaxing slightly the hypotheses on the smoothness of the Fano variety we get the following weaker definition.

\begin{definition}\label{whost}
A triangulated category $\cat{T}$ is \it weakly Fano-visitor \rm if there exists a (possibly singular) Fano variety $F$,
a categorical crepant resolution of singularities $\cat{DF}$ of $F$ and a fully faithful
functor $\cat{T} \to \cat{DF}$ such that $\cat{DF} = \langle \cat{T},\cat{T}^\perp \rangle$. Notice that this implies that
the functor $\cat{T} \to \cat{DF}$ has a right and left adjoint by definition of semiorthogonal decomposition.
As before, if $\cat{T}\cong \Db(X)$ for a smooth projective variety $X$, then $X$ itself is said to be \it weakly Fano-visitor \rm.
\end{definition}

\section{Homological Projective Duality for determinantal varieties}

We describe here homological projective duality for determinantal
varieties in terms of the Springer resolution of the space of
$n \times m$ matrices of rank at most $r$ and in terms of categorical
resolution of singularities.

\subsection{The Springer resolution of the space of matrices of bounded rank}\label{desingo}

Let us introduce the variety $\Zr$ of $n \times m$ matrices over our
base field, having rank at most $r$.
Let $U$, $V$ be vector spaces, 
with $\dim U=m$, $\dim V=n$, and assume $n \ge m$.
Set $W = U \tensor V$.
Let $r$ be an integer in the range $1 \le r \le m-1$.
We define $\kz^r=\Zr$ to be the variety of matrices $M : V
\to U\dual$ in $\PP W$
cut by the minors of size $r+1$ of the matrix of 
indeterminates:
\[
\psi = \begin{pmatrix}
  x_{1,1} & \ldots & x_{m,1} \\
  \vdots & \ddots & \vdots\\
  x_{m,n} & \ldots & x_{m,n}
\end{pmatrix}
\]

\subsubsection{Springer resolution and projective bundles} Consider the Grassmann variety $\GG(U,r)$ of $r$-dimensional
quotient spaces of $U$, the tautological sub-bundle and the quotient bundle
over $\GG(U,r)$, denoted respectively by $\ku$ and $\kq$, respectively
of rank $m-r$ and $r$. We will write $\GGG$ for $\GG(U,r)$.

The tautological (or Euler) exact sequence reads:
\begin{equation}
  \label{tautological}
0 \to \ku \to U \tensor \ko_{\GGG} \to \kq \to 0.  
\end{equation}

We will use the following notation:
\[
\zr = \PP(V \tensor \kq).
\]
However, the dependency on $m,n,r$ will often be omitted.

The manifold $\kx=\zr$ has dimension $r(n+m-r)-1$. It is the resolution of singularities of the variety of $m \times n$
matrices of rank at most $r$, in a sense that we will now review.
Denote by $p$ the natural projection $\kx \to \GGG$.
The space $H^0(\GGG,\kq)$ is naturally
identified with $U$.

Let us denote by $\ko_{\kx}(H_X)$ the relatively
ample tautological line bundle on $\kx$. We will often write simply $H$ for $H_X$. We get
natural isomorphisms:
\[
H^0(\GGG,V \tensor \kq) \simeq H^0(\kx,\ko_{\kx}(H)) \simeq W = U \tensor V.
\]
Therefore, the map $f$ associated with the linear system
$\ko_{\kx}(H)$ maps $\kx$ to $\PP W$, and
clearly $\ko_{\kx}(H) \simeq f^*(\ko_{\PP W}(1))$.
This is summarized by the diagram:
\[
\xymatrix@-2ex{
\kx \ar^-{f}[r] \ar_-{p}[d] & \PP W = \PP (U \tensor V) \\
\GGG
}
\]

On the other hand, we will denote by $P$ the pull-back to $\PP(V\otimes \kq)$ of the first Chern class $c_1(\kq)$ on
$\GGG$. Hence we have that $c_1(V\otimes 
 \kq)$ pulls-back to $n P$ and $\omega_{\GGG}$ to $-m P$. The Picard group of $\kx$ is generated by $P$ and $H$.

\medskip
 
 Notice that giving a rank-1 quotient of $W=U \tensor V$ corresponds to the choice of 
a linear map $M : V \to U\dual$, so 
an element of $\PP W$ can be considered as (the
proportionality class of) the linear map $M$.
On the other hand, the map $f$ sends a rank-1 quotient of $V
\tensor \kq$ over a point $\lambda \in \GGG$ to the quotient of $W$ obtained by composition
with the obvious quotient $U \to \kq_\lambda$.

Therefore, the matrix $M$ lies in the image of $f$ if and only if 
$M$ factors through $V \to \kq_\lambda\dual$, for some $\lambda \in
\GGG$, i.e., if and only if $\rk(M) \le r$.
Clearly, if $M$ has precisely rank $r$ then it determines $\lambda$
and the associated quotient of $U \to \kq_\lambda$.
Since this happens for a general matrix $M$ of $\kz^r=\Zr$, the map $f:
\kx \to \kz^r$ is birational. This map is in fact a desingularization, called the {\it Springer
resolution}, of $\kz^r$.
It is an
isomorphism above the locus of matrices of rank exactly $r$.

In a more concrete way, given $\lambda \in \GGG$ we let
$\pi_\lambda$ be the linear projection from $U\dual$ to $U\dual/\kq\dual_\lambda$.
Then, the variety $\kx$ can be thought of as: 
\[
\kx=\{(\lambda,M) \in \GGG \times \kz^r \mid \pi_\lambda\circ M  =0 \}.
\]

This way, the maps $p$ and $f$ are just the projections from $\kx$
onto the two factors.

\medskip

Let us now look at the dual picture.
We consider the projective bundle:
\[
\yr = \PP(V\dual \tensor \ku\dual).
\]

Write $\ky=\yr$ for short. Denote by $q$ the projection $\ky \to
\GGG$. We will denote by $H_Y$ (or sometimes just by $H$) the tautological ample line bundle on
$\ky$.
This time, since $H^0(\GGG,\ku\dual) \simeq U\dual$, the linear system
associated with $\ko_{\ky}(H)$ sends $\ky$ to $\PP W^\vee \simeq \PP(V\dual \tensor U\dual)$ via a map that we call $g$. 
By the same argument as above, $g$ is a desingularization of the
variety $\kw^{r}$ of matrices $V^\vee \to U$ in $\PP W^\vee$ of corank at least $r$.
There exists an obvious isomorphism  $\kz^{m-r} \cong \kw^{r}$,
which we will use without further mention. 

The spaces $\PP W$ and $\PP W^\vee$ are equipped with tautological
morphisms of sheaves, which are both identified by the the matrix
$\psi$, corresponding to the identity in $W \tensor W\dual = U \otimes V \otimes
U\dual \otimes V\dual$:
\begin{align}
\label{tauto2} &V \otimes \ko_{\PP W}(-1) \stackrel{\psi}{\to} U\dual
\otimes \ko_{\PP W}, \\
\label{tauto} &V\dual \otimes \ko_{\PP W ^\vee}(-1) \stackrel{\psi}{\to}
U \otimes \ko_{\PP W^\vee}.
\end{align}

\begin{definition}
We will denote by $\kf$ and $\ke$, the cokernel of the
tautological map appearing in Eq. \eqref{tauto2}, respectively Eq. \eqref{tauto}.
\end{definition}

\begin{lemma} \label{un peu penible}
  We have isomorphisms  $\kx \simeq \GG(\kf,m-r)$ and $\ky \simeq \GG(\ke,r)$.
\end{lemma}

\begin{proof}
  We work out the proof for $\ky$, the argument for $\kx$ being identical.
  Given a scheme $S$ over our field, an $S$-valued point $[e]$ of $\GG(\ke,r)$
  is given by a morphism $s : S \to \PP W^\vee$ and the equivalence class of
  an epimorphism $e : s^* \ke \to \kv$, where $\kv$ is locally free of
  rank $r$ on $S$.
  On the other hand, an $S$-point $[y]$ of $\ky$ corresponds to a
  morphism $t : S \to \GGG$ together with the class of a quotient $y : V\dual
  \tensor t^* \ku \dual \to \kl$, with $\kl$ invertible on $S$. In turn,
  $t$ is given by a locally free sheaf of rank $r$ on $S$ and a
  surjection from $U \tensor \ko_S$ onto this sheaf.

  Given the point $[e]$, we compose $e$ with the surjection $U \tensor
  \ko_S \to s^* \ke$ and denote by $t_e$ the resulting map $U \tensor
  \ko_S \to \kv$. This way, $t_e$ provides the required morphism $t : S \to
  \GGG$, and clearly $t^* \kq \simeq \kv$, so the kernel of $U
  \tensor  \ko_S \to \kv$ is just $t^* \ku$.
  Clearly, we have $t_e \circ s^* \psi = 0$ so that $s^* \psi$
  factors through a map $V\dual \tensor \ko_S(-1) \to t^* \ku$.
  Giving this last map is equivalent to the choice of a map 
  $V\dual \tensor t^*\ku\dual \to \ko_S(1)$, which we define to be the point
  $[y]$ associated with $[e]$.

  Conversely, let $t$ be represented by a locally free sheaf $\kv=t^* \kq$ of
  rank $r$ on $S$ and by a quotient $U \tensor \ko_S \to \kv$, whose
  kernel is $t^* \ku$.
  Then, given point $[y]$ and the quotient $y$, we consider the
  composition of $y$ and $U\dual \tensor \ko_S \to \ku\dual$ to obtain a quotient
  $s_y : V\dual \tensor U\dual \to \kl$. This gives the desired morphism $s
  : S \to \PP W^\vee$.
  Moreover,  the
  map $V\dual \tensor \ko_S \to t^* \ku \tensor \kl$ associated with $y$
  can be composed with
  the injection $t^* \ku \tensor \kl \to U \tensor \kl$ to get a map 
  $V\dual \tensor \ko_S \to U \tensor \kl$, or equivalently $V\dual \tensor
  \kl\dual \to U \tensor \ko_S$, and this map is nothing but $s^* \psi$.
  Of course, composing this map with the projection $U \tensor \ko_S
  \to t^*\kq = \kv$ we get zero, so there is an induced surjective map
  $s^* \ke \to \kv$. We define the class of this map to be the point $[e]$
  associated with $[y]$.

  We have defined two maps from the sets of $S$-valued points of our
  two schemes, which are inverse to each other by construction. The
  lemma is thus proved.
\end{proof}

\subsubsection{Linear sections and projectivized sheaves}\label{proshe}

Let now $c$ be an integer in the range $1 \le c \le m n$, and suppose we a have $c$-dimensional
vector subspace $L$ of $W$:
\[
L \subset U \tensor V = W.
\]

We have thus the linear subspace $\PP_L \subset \PP W$ of
codimension $c$, defined by $\PP_L=\PP(W/L)$.
Dually, we have a linear subspace $\PP^L=\PP L^\perp$ of dimension $c-1$
in $\PP W\dual$, whose defining equations are the elements of $L^\perp
\subset W\dual$.
We define the varieties:
\[
X^r_L = \zr \times_{\PP W} \PP_L, \qquad Y^r_L = \yr \times_{\PP W^\vee} \PP^L.
\]

We also write:
\[
Z_L^r = \kz^{r}_{m,n} \cap \PP_L, \qquad Z_r^L = \kz^{m-r}_{m,n} \cap \PP^L.
\]
We will drop $r$, $n$ and/or $m$ from the notation when no confusion is possible.
We will always assume that $L \subset W$ is an {\it admissible
  subspace} in the sense of \cite{kuznetsov:hpd},
which amounts to ask that $X_L$ and $Y_L$ have expected
dimension. This means that we have:
\begin{align*}
\dim Z_L = \dim X_L & = \dim \zr - c = r(n+m-r)-c-1 \\  
\dim Z^L  = \dim Y_L & = \dim \yr - (mn-c) = r(m-n-r)+c-1.
\end{align*}

Let us now give another interpretation of the choice of our linear subspace $L \subset W$.
To this purpose we consider the Grassmann variety $\GG(V,r)$ with the
its tautological rank-$r$ quotient bundle which we denote by $\kt$.
Dually, we consider $\GG(V^\vee,m-r)$ and denote by $\ks^\vee$ the
tautological quotient bundle of rank $m-r$.
Observe that there are natural isomorphisms:
\begin{align*}
 L\dual \tensor W  =  L\dual \tensor U \tensor V  &
  \simeq  \Hom(L \tensor \ko_{\GGG}, V \tensor \kq) \simeq \\
  & \simeq  L\dual \tensor H^0(\kx,\ko_{\kx}(H) )\simeq \\
  & \simeq  \Hom(L \otimes \ko_{\GG(V,r)}, U\otimes \kt) \nonumber.
\end{align*}
There are similar isomorphisms for $\GG(V^\vee,m-r)$.
We denote by $s_L$ the global section of $L\dual \tensor
H^0(\kx,\ko_{\kx}(H))$ corresponding to $L \subset W$
via these isomorphisms.
The subspace $L$ corresponds also to morphisms of bundles on the
Grassmann varieties, which we write as:
\[
M_L : L \tensor \ko_{\GGG} \to V \tensor \kq, \qquad 
N_L : L \otimes \ko_{\GG(V,r)} \to U\otimes \kt
\]
We also write:
\[
M^{L} : L^\perp \tensor \ko_{\GGG} \to V\dual \tensor \ku\dual, \qquad 
N^{L} : L^\perp \otimes \ko_{\GG(V^\vee,m-r)} \to U\dual\otimes \ks\dual
\]
for the morphisms corresponding to $L^\perp \subset U\dual \tensor V\dual$.

\begin{proposition} \label{mi sa che bisogna dimostrarla}
  We have the following equivalent descriptions of $X_L$:
 \begin{enumerate}[(i)]
  \item \label{e uno} the vanishing locus $\VV(s_L)$ of the section $s_L \in L\dual \tensor
H^0(\kx,\ko_{\kx}(H))$;
  \item \label{e due} the projectivization of $\coker(M_L)$;
  \item \label{e tre} the projectivization of $\coker(N_L)$;
  \item \label{e quattro} the Grassmann bundle $\GG(\kf|_{\PP_L},m-r)$.
  \end{enumerate}

Dually, the variety $Y_L$ is:
\begin{enumerate}[(i)]
\item  the vanishing locus of the section $s^L \in
  (W/L) \tensor H^0(\kx,\ko_{\kx}(H))$;
  \item the projectivization of $\coker(M^L)$; 
  \item the projectivization of $\coker(N^L)$;
  \item the Grassmann bundle $\GG(\ke|_{\PP^L},r)$.
\end{enumerate}
\end{proposition}

\begin{proof}
We work out the proof for $X_L$, the dual case $Y_L$ being
analogous. First recall that the map $\kx \to \PP W$ is defined by
the linear system $\ko_{\kx}(H)$, while the inclusion $\PP_L \subset
\PP W$ corresponds to the projection $W \to W/L$. Hence the fibre product defining $X_L$
is given by the vanishing of the global sections in
$H^0(\kx,\ko_{\kx}(H))$ which actually lie in $L$, \it i.e. \rm by the vanishing of $s_L$,
so \eqref{e uno} is clear.

For \eqref{e due} we use essentially the same proof of Lemma \ref{un peu
  penible}. Indeed, given a scheme $S$ over our field, an $S$-valued point of
$\PP (\coker(M_L))$ is defined by a morphism $t : S \to \GGG$
together with the isomorphism class of a
quotient $y : t^* (\coker(M_L)) \to \kl$, with $\kl$ invertible on $S$.
On the other hand, an $S$-valued point of $X_L$ is given by a morphism $s : S
\to X_L$.
Once given $s$, composing with $X_L \to \kx \to \GGG$ we obtain
the morphism $t$. By the definition of $\kx$ as projective
bundle, together with $t$ we get a map $V \tensor t^*\kq \to \kl$ 
with $\kl$ invertible on $S$. This map composes to zero with $t^*(M_L) : L
\tensor \ko_S \to V \tensor t^*\kq$ since the image of $s$ is contained in $X_L$, hence in the
vanishing locus of
the linear section $s_L$.
Therefore this map factors through $t^*(\coker(M_L))$ and provides the
quotient $y$. It is not hard to check that  this procedure can be
reversed, which finally proves \eqref{e
due}.

The statement \eqref{e tre} is proved in a similar fashion, while
\eqref{e quattro} is just Lemma  \ref{un peu
  penible}, restricted to $\PP_L$.
\end{proof}

\subsection{The noncommutative desingularization}
In \cite{buch-leu-vdbergh1,buch-leu-vdbergh2}, noncommutative resolutions
of singularities for the affine cone over $\kz^r=\kz^{r}_{m,n}$ are constructed.
This is done by considering the vector bundles $V \otimes \kq$ instead
of their projectivization, and Kapranov's strong exceptional collection on the Grassmannian
\cite{kapra-grassa} (for the details see \cite{buch-leu-vdbergh2}). Here we carry on this construction to the projectivized
determinantal varieties.

Consider $\kx=\zr$ as rank$-(r n-1)$ projective bundle $p:\kx \to \GGG$. Orlov
\cite{orlov:proj_bundles} gives a semiorthogonal decomposition
\begin{equation}\label{eq:lefschetz-for-Xtondo}
\Db(\kx)= \langle p^*\Db(\GGG), \ldots, p^*\Db(\GGG)((r n-1)H)\rangle.
\end{equation}
On the other hand, Kapranov
shows that $\GGG$ has a full strong exceptional collection \cite{kapra-grassa} consisting of vector bundles.
We obtain then an exceptional collection on $\kx$ consisting of vector bundles, and hence a tilting bundle
$E$ as the direct sum of the bundles from the exceptional collection.
 Let us consider $M:=Rf_*E$, and let $\kr:=\kEnd(E)$ and $\kr':=\kEnd(M)$ (where
 $\kEnd$ denotes the sheaf of endomorphisms).

\begin{proposition}\label{prop:noncomm-resol-general}
The endomorphism algebra $\kEnd(M)$ is a coherent $\ko_{\kz^r}$-algebra
Morita-equivalent to $\kr$. In particular, $\Db(\kz^r,\kr) \simeq \Db(\kx)$
is a categorical resolution of singularities, which is crepant if $m=n$.
\end{proposition}

\begin{proof}
First of all, since $\GGG$ has a strong full exceptional collection,
we have a tilting bundle $G$ over it. 
A cohomological calculation, together with the semiorthogonal decomposition \eqref{eq:lefschetz-for-Xtondo}
provides a tilting bundle $E= \bigoplus_{i=0}^{nr-1} p^* G \otimes \ko_{\kx}(iH)$ over $\kx$.
We have thus:
\[
\Db(\kx) \simeq \Db(\mathrm{End} (E)).
\]
Since the
exceptional locus of $f$ has codimension greater than one, 
\cite[Lemma 4.2.1]{vdb:flops} implies that $f_* \kr$ is reflexive.
There is a natural map $f_*\kr \to \kr'$ of reflexive sheaves
which explicitly reads:
\[
f_*\kr = \bigoplus_{i,j = 0}^{nr-1} f_*\kEnd(p^*G)(i-j) \to
\bigoplus_{i,j = 0}^{nr-1} \kEnd(f_*p^*G)(i-j) = \kr'.
\]
Again, since the
exceptional locus of $f$ has codimension greater than one the locus where $f_*\kr$ and
$\kr'$ may be non-isomorphic has codimension at least 2. Since both sheaves are
reflexive, we obtain $f_*\kr \cong \kr'$ (compare with \cite[Proposition 6.5]{buch-leu-vdbergh1}).
Moreover we know from
\cite[Proposition 3.4]{buch-leu-vdbergh2} that $R^kf_* \kr = 0$ for $k>0$ so
we actually have:
\[
Rf_*\kr \cong \kr'.
\]
Therefore:
\[
\mathrm{End}(E) \simeq  H^\bullet(\kr) \simeq H^\bullet(Rf_* \kr) \simeq H^\bullet(\kr'). 
\]
We have now proved:
\[
\Db(\kx) \simeq \Db(\mathrm{End}(E)) \simeq \Db(H^\bullet(\kr'))
\simeq \Db(\kz^r,\kr').
\]

Finally, $\kr'$ is maximally Cohen-Macaulay by
\cite[Proposition 3.4]{buch-leu-vdbergh2} (as this property is local) and 
has finite global dimension since it is Morita-equivalent to the endomorphism algebra $\kr$,
which is defined over a smooth variety. If $m=n$, the variety $\kz^r$ has Gorenstein
singularities and $f$ is a crepant resolution, so that the noncommutative resolution is also
crepant (compare with \cite{buch-leu-vdbergh1}).
\end{proof}

\subsection{Homological projective duality for matrices of bounded rank}

With this in mind, we can prove our main result directly from
Kuznetsov's HPD for the projective bundles $\zr=\kx$ and 
$\yr=\ky$. We consider the rectangular Lefschetz decomposition
\eqref{eq:lefschetz-for-Xtondo} for $\kx$ with respect to $\ko_\kx(H)$.

\begin{theorem} \label{thm:main1}
The morphism  $g :
\ky \to \PP W^\vee$ is the homological projective dual of $f : \kx \to \PP W$, relatively over $\GGG$, with respect to
the rectangular Lefschetz decomposition \eqref{eq:lefschetz-for-Xtondo} induced by $\ko_\kx(H)$, generated by $n r
{m\choose r}$ exceptional bundles.
\end{theorem}

\begin{proof}
Given the setup of \S \ref{desingo}, we consider the vector bundles $V
\tensor \kq$ and $V^\vee \tensor \ku^\vee$ over $\GGG$ and recall that
$\kx=\PP(V
\tensor \kq)$ and $\ky=\PP(V^\vee \tensor \ku^\vee)$.

Set $\cat{A}= p^*(\Db(\GGG))$. The decomposition \eqref{eq:lefschetz-for-Xtondo}
of the projective bundle $\kx \to \GGG$ then reads:
\[
\Db(\kx) = \langle \cat{A}, \cat{A}(H), \ldots, \cat{A}((r n-1)H) \rangle.
\]
This is a rectangular Lefschetz decomposition with respect to
$\ko_{\kx} (H)$, generated by $nr$ copies of Kapranov's exceptional
collection on $\GGG$, hence by $n r
{m\choose r}$ exceptional bundles.

Clearly the vector bundles $V
\tensor \kq$ and $V^\vee \tensor \ku^\vee$  are generated by their global sections, so we may apply
apply \cite[Corollary 8.3]{kuznetsov:hpd} to their projectivization (actually we use the Grothendieck's notation for
projectivized bundles rather than the usual notation as in
\cite{kuznetsov:hpd}, but this does affect the result).
 The evaluation
map of global sections of $V \tensor \kq$ gives
\eqref{tautological} tensored with the identity over $V$ i.e.:
\[
0 \to V \tensor \ku \to W \to V \tensor \kq \to 0.
\]

This says that $V^\vee \tensor \ku^\vee$ is the orthogonal in Kuznetsov's sense
of $V \tensor \kq$. Also, the morphism associated with the
tautological line bundle $H_X$ over $\kx$ is $f$, while $g$ is
associated with $H_Y$ over  $\ky$.
Therefore \cite[Corollary 8.3]{kuznetsov:hpd} applies and gives the result.

Note that $\Db(\ky)$ is generated by $n (m-r) {m \choose r}$
exceptional vector bundles.
\end{proof}

We can rephrase this in terms of categorical resolutions, as
a consequence of
Proposition \ref{prop:noncomm-resol-general}. In this way, one can
state HPD as a duality
between categorical resolutions of determinantal varieties given by matrices of fixed rank and corank.
This leads us to prove our second main Theorem.

\begin{theorem}\label{thm:main2}
There is a $\ko_{\kz^r}$-algebra $\kr'$ such that $(\kz^r, \kr')$ is a categorical resolution
of singularities of $\kz^r$. Moreover, $\Db(\kz^{r}, \kr') \simeq \Db(\kx)$ so that
$(\kz^{r}, \kr')$ is HP-dual to $(\kz^{m-r},\kr')$.
\end{theorem}

\begin{proof}
Recall that $\kx$ is a projective bundle over a Grassmann variety, and hence has a full exceptional sequence.
By applying Proposition \ref{prop:noncomm-resol-general} to the full exceptional sequence on $\kx$, we get the first statement. The second statement is now straightforward
from Theorem \ref{thm:main1}, together with the isomorphism $\kx^{m-r}\simeq \ky^r$.
\end{proof}

\subsection{Semiorthogonal decompositions for linear
  sections}\label{dimensions}

Let $L$ be a dimension $c$ subspace of $U\tensor V=W$, given by the
choice of an element $t \in L\dual \otimes W$. Recall that we assume
that the subspace $L \subset W$ is {\it admissible} in
the sense of \cite{kuznetsov:hpd}.
This happens if $L$ is general enough in $W$.

Moreover, again if $L$ is general enough,  the
singularities of $Z_L=Z_L^r$ appear precisely along $Z_L^{r-1}$.
Also, the map $f$, for the rank $r$ locus, is an isomorphism when restricted to $Z_L \setminus Z_L^{r-1}$. Furthermore, we recall from the preceding
section that $Z_L$ is a determinantal variety inside $\PP^{mn-c-1}$ given by a $m \times n$ matrix
of linear forms and $\Db(Z_L, \kr'_{\PP_L})$ is a categorical resolution of singularities of $Z_L$, where $\kr'_{\PP_L}$ is the pull-back of $\kr'$ from $\kz_{m,n}^r$
to $Z_L$ under the natural restriction map.

Notice that if $Z_L$ is smooth, then $\Db(Z_L) \simeq \Db(X_L)$, in fact, $Z_L \simeq X_L$ in this case.
Similarly, if $Z^L=Z^L_{r}$ is smooth, then $\Db(Z^L) \simeq \Db(Y_L)$
as again $Z^L \simeq Y_L$ in this case. In particular, in the smooth case, the sheaves of algebras $\kr'_{\PP_L}$ are Morita-equivalent to
the structure sheaf.

Our goal now is to draw consequences from the homological projective duality that we have displayed. Notably we will give in several examples a positive
answer to the questions asked in the introduction, \it i.e. \rm Bondal's Question \ref{qu:bondal} and question \ref{qu:catrep=ration} concerning rationality and categorical representability. Remember that $\kx$ (respectively $\ky$) is the
projectivization of a vector bundle of rank $nr$ (resp. $n(m-r)$) over
$\GGG$. Hence, by Orlov's result (\cite{orlov:proj_bundles}) on
the semiorthogonal decompositions for projective bundles we have:
\begin{eqnarray}
\Db(\kx)& =& \langle \cat{A}, \cat{A}(H), \ldots, \cat{A}((nr-1)(H) \rangle;\nonumber\\
\nonumber \Db(\ky) & = & \langle \cat{B}((1-nm+nr)H),  \ldots, \cat{B}(-H),\cat{B} \rangle,
\end{eqnarray}
where $\cat{A}$ and $\cat{B}$ are the respective pull-backs of $\Db(\GGG)$ to the projective bundles.
This in turn implies that, via HPD, when we intersect $\kx$ with $\PP_L$ and $\ky$ with $\PP^L$, we have the following 

\begin{eqnarray}
\Db(X_L)& =& \langle \cat{C}_L,\cat{A}(H), \ldots, \cat{A}(nr-c)(H) \rangle;\nonumber\\
\nonumber \Db(Y_L) & = & \langle \cat{B}((-c+nr)H),  \ldots, \cat{B}(-H),\cat{C}_L
\rangle. 
\end{eqnarray}

Recalling that $\Db(X_L) \simeq \Db(Z_L^{r},\kr'_{\PP_L})$ and $\Db(Y_L)=\Db(Z^L_r,\kr'_{\PP^L})$ are categorical resolutions of
singularities of dual determinantal varieties, we get: 
\begin{eqnarray}
\Db(Z^{r}_L,\kr'_{\PP_L})& =& \langle \cat{C}_L,\cat{A}(H), \ldots, \cat{A}(nr-c)(H) \rangle;\nonumber\\
\nonumber \Db(Z_r^L,\kr'_{\PP^L}) & = & \langle \cat{B}((-c+nr)H),  \ldots, \cat{B}(-H),\cat{C}_L \rangle.
\end{eqnarray}

Finally, the categories $\cat{A}$ and $\cat{B}$ are both generated by ${m \choose r}$ exceptional objects.

\begin{corollary}\label{cor:functors}
Suppose that $L \subset W$ is admissible of dimension $c$.
\begin{enumerate}[(i)]
\item If $c>nr$, there is a fully faithful functor 
$$\Db(Z_L,\kr'_{\PP_L}) \simeq  \Db(X_L) \longrightarrow \Db(Y_L) \simeq \Db(Z_{r}^L, \kr'_{\PP^L})$$
whose orthogonal complement
is given by $c-nr$ copies of $\Db(\GGG)$, and is then generated by
$(c-nr){m \choose r} $ exceptional objects.
\item If $nr=c$, there is an equivalence 
$$\Db(Z_L,\kr'_{\PP_L}) \simeq  \Db(X_L) \simeq \Db(Y_L) \simeq \Db(Z_{r}^L, \kr'_{\PP^L})$$
\item If $c<nr$, there is a fully faithful functor
$$\Db(Z_{r}^L,\kr'_{\PP^L}) \simeq \Db(Y_L) \longrightarrow \Db(X_L) \simeq \Db(Z_L,\kr'_{\PP_L})$$
whose orthogonal complement is given by $nr-c$ copies of $\Db(\GGG)$, and is then generated by ${m \choose r}(nr-c)$ exceptional objects.
\end{enumerate}

\end{corollary}

\begin{proof}
The statement is obtained applying Kuznetsov's Theorem \ref{thm:HPD}
  to the pair of HPDual varieties from Theorem \ref{thm:main1}, and using the resolutions of singularities
described in Theorem \ref{thm:main2}. The functors involved can be explicitly described as Fourier--Mukai
with kernels in $\Db(X_L \times Y_L)$ (see the detailed description in
the original Kuznetsov's paper \cite[\S 5]{kuznetsov:hpd}).
\end{proof}

Using the notation introduced in Section \ref{desingo} for the generators of the Picard group, we have the following formula for the canonical bundle of
 $\kx$:
\[
\omega_{\kx} \simeq  \ko_{\kx}(-nr H +(n-m)P).
\]

A consequence of this formula is the following lemma. Call $\phi_K$
the canonical map of $X_L$, i. e. the rational map associated with the
linear system $|\omega_{X_L}|$. Write also $\phi_{-K}$ for the map
associated with $|\omega_{X_L}^\vee|$ (i.e. the anticanonical map).

\begin{lemma} \label{adjuncion}
The canonical bundle of the linear section $X_L$ is:
\[
\omega_{X_L} \simeq  \ko_{X_L}((c-nr) H + (n-m)P).
\]
\begin{enumerate}[i)]
\item \label{when CY} The variety $X_L$ is Calabi-Yau if and only if $m=n$ and $c=nr$. 

\item \label{when positive} If $c>nr$, or if $c=nr$ and $n > m$, $\phi_K$ is a
  birational morphism onto its image.

\item \label{when negative} If $c < nr$ and $m=n$, $\phi_{-K}$ is a
  birational morphism onto its image. If moreover
  $X_L^{r-1}=\emptyset$, $\phi_{-K}$ is an embedding and $X_L$ is Fano.
\end{enumerate}
\end{lemma} 

\begin{proof}
The formula for $\omega_{X_L}$ is obvious by adjunction. By this
formula, $\omega_{X_L}\simeq \ko_{X_L}$ whenever $m=n$ and
$c=nr$. Conversely, remark that $X_L$ is connected, so if $X_L$ is CY, then 
there is no nontrivial semiorthogonal decomposition of $\Db(X_L)$. Corollary
\ref{cor:functors} forces then $c \geq nr$.

Suppose $c > nr$, or $c=nr$ and $n>m$. Notice first that both $P$ and $H$ are nef.
Then canonical divisor is a linear
combination of nef divisors with positive coefficients, which is in turn nef. 
On the other hand, we have that $\omega_{X_L}$ is $\ko_{X_L}$ if $c=nr$ and $m=n$, 
so using $c \geq nr$ we conclude the proof of \eqref{when CY}.

For \eqref{when positive}, by definition $H$ and $P$ are
base-point-free and $H$ is very ample away from the exceptional locus
of $f$, so the statement follows directly from the formula for $\omega_{X_L}$. A similar argument proves
\eqref{when negative}.
\end{proof}

\begin{corollary}\label{cor:adjunction}
We have the following formulas for the canonical bundles
\begin{align*}
&\omega_\ky \simeq \ko_\ky(-n(m-r) H + (n-m)Q), \\
&\omega_{Y_L} \simeq  \ko_{Y_L}((nr-c) H + (n-m)Q).
\end{align*}

In particular, $Y_L$ is Calabi-Yau if and only if $m=n$ and $c=nr$. If
$c<nr$, or if $c=nr$ and $n > m$, then the canonical map of $Y_L$ is a
birational morphism onto its image.
If $c > nr$, $m=n$ and $Y_L^{r-1}=\emptyset$ then $Y_L$ is Fano.
\end{corollary}

\begin{proof}
Everything follows from the isomorphism $\ky^{r} \simeq
\kx^{m-r}$. Indeed, we find $Y^r_L \simeq X_{L^\perp}^{m-r}$: and,
since $\dim L^\perp + \dim L = \dim (W) = nm$,
we get the formula for $\omega_{Y_L}$ from Lemma \ref{adjuncion}, recalling that the relative hyperplane section is identified with $Q$ in this case.
The other statements follow as in Lemma \ref{adjuncion}.
\end{proof}

We resume in Table \ref{table:all-r} the results of this section. The
functor mentioned there is the HPD functor.
\begin{small}
\begin{table}[h!]
{
\begin{center}
\begin{tabular}{|l||l|l|l|}
\hline
& $c<nr$ & $c=nr$ & $c>nr$  \\
\hline
\hline
HPD Functor & $\Db(Y_L) \to \Db(X_L)$ & equivalence & $\Db(X_L) \to \Db(Y_L)$  \\
\hline\hline
\multirow{2}{*}{$Y_L \to Z^L$} & nef canonical  & nef canonical if $n\neq m$ & \\
\cline{2-4} & Fano visitor if $n=m$ & CY if $n=m$ &  Fano if $n=m$\\
\hline\hline
\multirow{2}{*}{$X_L \to Z_L$ }& & nef canonical if $n\neq m$ & nef canonical \\
\cline{2-4} & Fano if $n=m$ & CY if $n=m$ & Fano visitor if $n=m$ \\
\hline
\end{tabular}
\end{center}
}
\caption{Behaviour of HPD functors according on $c$ and $nr$.}
\label{table:all-r}
\end{table}
\end{small}

\section{Birational and equivalent linear sections}\label{sect:CY}

As explained in Corollary \ref{cor:functors} and then displayed in Table \ref{table:all-r}, the condition $c=nr$
guarantees that HPD gives an equivalence of categories. Hence our
construction gives examples of derived equivalences of Calabi-Yau
manifolds for any $n=m$. One first  example was
produced in \cite{vivalafisica}.
In fact the authors of \cite{vivalafisica} take $n=m=4$, $r=2$, the self dual
orbit of rank 2, $4\times 4$ matrices and consider the codimension eight threefolds
obtained by taking orthogonal linear sections in $\PP^{15}$. 
In fact, our construction shows that these two Calabi-Yau are derived equivalent. On the other hand it is very likely that they are one the flop of the other. We can show indeed that $X_L$ and $Y_L$ are birational whenever $c=nr$.

Assume now that $c=nr$. 
Remark that the two vector bundles appearing in the map $M_L$ of Proposition \ref{mi sa
  che bisogna dimostrarla} have the same rank, namely $nr$.
Let us denote by $D_L$
the hypersurface in $\GGG$  defined by the vanishing of determinant of $M_L$:
\[
M_L : L \tensor \ko_{\GGG} \to V \tensor \kq.
\]

The degree of $D_L$ is $n$. Dually, we write $D^L$ the hypersurface in $\GGG$
whose equation is  the determinant of: 
\[
M^{L} : L^\perp \tensor \ko_{\GGG} \to V\dual \tensor \ku\dual.
\]

\begin{proposition}
If $c=nr$ then $D^L=D_L$, and $X_L$ is birational to $Y_L$.
\end{proposition}

\begin{proof}
To see this, we write the following  exact commutative diagram:
\[
\xymatrix@-3ex{
&& 0 \ar[d] & 0 \ar[d] & \\
&& V\dual \tensor \kq\dual \ar[d] \ar@{=}[r] & V\dual \tensor \kq\dual \ar^-{(M_L)^*}[d]\\
0 \ar[r] &L^\perp \tensor \ko_{\GGG} \ar@{=}[d] \ar[r] & U\dual \tensor V\dual \tensor\ko_{\GGG}
\ar[d] \ar[r] & L\dual \tensor \ko_{\GGG} \ar[r] \ar[d] & 0\\ 
 & L^\perp \tensor \ko_{\GGG} \ar^-{M^L}[r] & V\dual \tensor \ku\dual \ar[d] \ar[r] & \kk \ar[d] \ar[r] & 0\\
 & & 0 & 0
}
\]
Here, $\kk$ is the cokernel both of $M^L$ and of $(M_L)^*$. This says that:
\[
D^L=\VV(\det(M^L))=\VV(\det(M_L\dual))=\VV(\det(M_L))=D_L.
\]

Now let us look at $X_L$ and $Y_L$. The sheaf $\kk$ is supported
on $D=D^L$, and is actually of the form $\iota_*(\kk_r)$, where $\kk_r$
is a reflexive sheaf of rank $1$ on $D$ and $\iota : D \to \GGG$
is the natural embedding. The
cokernel of $M_L$ is also of the form $\iota_*(\kk^r)$, with $\kk^r$
reflexive of rank $1$ on $D$.
By Grothendieck duality, since $D$ has degree $n$, the previous diagram says that $\kk^r \simeq \kk_r^\vee(n)$.
 On the (open and dense) locus of $D$ where $\kk^r$ and $\kk_r$
are locally free, the variety $D$ coincides with $X_L$ and
$Y_L$. Therefore, by Proposition \ref{mi sa che bisogna dimostrarla}, these varieties are both birational to $D$.
\end{proof}

A priori, $X_L$ is not isomorphic to $Y_L$, as the
projectivization of the two sheaves $\kk_r$ and $\kk_r^\vee$ gives in
principle non-isomorphic varieties (cf. Example \ref{noniso} below).
This does not happen if $\kk_r$ is locally free of rank $1$ on $D$,
which in turn is the case if $D$ is smooth. Also, when the
singularities of $D$ are isolated points, then in order for
$\PP(\kk_r)$ to be isomorphic to
$\PP(\kk_r^\vee)$, it suffices to check that the rank of $\kk_r^\vee$ and $\kk_r$ is the same
at those points, and this is of course true. Then we have:

\begin{remark}
Suppose that $D$ is smooth or has isolated singularities, then $X_L$ is isomorphic to $Y_L$.
\end{remark}

If we assume that $X_L$ is Calabi-Yau, then $m=n$ and $c=nr$ so we are in a sub-case of our description above, and birationality still holds. Thus, in dimension 3, the derived equivalences would follow also from the work of Bridgeland \cite{bridg-flop}.

\begin{example} \label{noniso}
  Let us describe an example of two determinantal varieties $X_L$ and
  $Y_L$ which are derived equivalent, birational, but not
  isomorphic. Actually one can describe infinitely many examples this
  way, all of dimension at least $5$. In all of them the
  canonical system is birational onto a hypersurface of general type
  in $\GGG$.

  Take $(r,m,n)=(3,5,7)$, $c=21$ and consider a general subspace $L
  \subset W$. Then $X_L$ and $Y_L$ are both smooth
  projective $5$-folds.
  The Picard group $\Pic(X_L)$ is isomorphic to $\ZZ^2$,
  generated by (the restriction of) $H_X$ and $P$, while $\Pic(Y_L)$ is
  also isomorphic to $\ZZ^2$, generated by $H_Y$
  and $Q$. Note that $\omega_{X_L} \simeq \ko_{X_L}(2P)$ while
  $\omega_{Y_L} \simeq \ko_{Y_L}(2Q)$.

  We claim that $X_L$ and $Y_L$ are not isomorphic.
  Indeed, if there was an isomorphism $f:X_L
  \to Y_L$, we should have $f^*(Q) = P$ because of the expression of
  the canonical bundle. Since $(f^*(Q),f^*(H_{Y_L}))$
  should form a $\ZZ$-basis of $\Pic(X_L)$, we have $f^*(H_{Y_L})=H_{X_L}+aP$,
  for some $a \in \ZZ$. But a straightforward computation shows that
  $(H_{X_L}+aP)^5$ is never equal to $H_{Y_L}^5$, for any choice of
  $a$. 

  So the $5$-folds $X_L$ and $Y_L$ are not isomorphic. They are however derived
  equivalent via HPD and both birational by projection to a determinantal hypersurface $D$
  of degree $7$ in 
  $\GG(5,3)$. The canonical bundle of this hypersurface is
  $\ko_{D}(2)$.
  The determinantal model of $X_L$ (respectively, of $Y_L$) is the
  fivefold of degree $490$ (respectively, $1176$)
  cut in $\PP^{13}$ (respectively, in $\PP^{20}$)
  by the $4\times 4$ minors (respectively, the $3\times 3$ minors) of
  a sufficiently general $5\times 7$ matrix of linear
  forms. 
\end{example}

Concerning rationality of determinantal varieties, we have the
following result.

\begin{proposition} \label{razionalita}
The variety $X_L$ is rational if $nr > c$;  $Y_L$ is rational if $c>nr$.\end{proposition}  

\begin{proof}
  By Proposition \ref{mi sa che bisogna dimostrarla},  $Y_L$ is the 
  projectivization of the cokernel sheaf of the map $M^L$.
  Recall that $\dim(L^\perp)=nm-c$ and that $V^\vee \otimes \ku^\vee$
  has rank $n(m-r)$.

  So if $c>nr$, i.e. if $n(m-r)>nm-c$, there is a Zariski
  dense open subset of $\GG(U,r)$ where $M^L$ has constant rank
  $mn-c$. 
  Hence an open piece 
  of $Y_L$ is the projectivization of a locally free
  sheaf over a rational variety, so $Y_L$ is rational.

  The same argument works for $X_L$.
\end{proof}

A side remark is that, using 
$N^L$ instead of $M^L$ we would rationality of $Y_L$ if
$c>m(n-m+r)$. However, one immediately proves that $m(n-m+r) \ge nr$.

\section{The Segre-determinantal duality}\label{sect:segre-det}

In this section, we give a more detailed description of the case $r=1$ (we suppress 
 $r$ from our notation for this section).
In this case $\kx \simeq \PP^{n-1}\times\PP^{m-1}$ is just a product
of two projective spaces and 
$X_L$ is a linear section of a Segre variety. 
On the other hand, $\ky$ is the Springer desingularization of
the space of degenerate matrices.

For this section and the following ones, we make use of the standard notation
$\ko_{X_L}(a,b)$ for the restriction to $X_L$ of $\ko_{\PP^{n-1}}(a)
\boxtimes \ko_{\PP^{m-1}}(b)$, so that $\ko_{X_L}(1,1)=\ko_{X_L}(H)$ and  
$\ko_{X_L}(0,1)=\ko_{X_L}(P)$. Proposition \ref{mi sa che bisogna
  dimostrarla} and Lemma \ref{adjuncion} become:

\begin{corollary}\label{descrXL}
The variety $X_L$ can be described in two following ways:
\begin{enumerate}[(i)]

\item as the projectivization of the cokernel of:
\[
  L \otimes \ko_{\PP U}(-1) \to V \otimes \ko_{\PP U};    
\]

\item as the projectivization of the cokernel of:
\begin{equation*}
  L \otimes \ko_{\PP V}(-1) \to U \otimes \ko_{\PP V}.
\end{equation*}
\end{enumerate}
Also, we have the formulas for the canonical bundle:
\begin{align*}
  \omega_{X_L} & = \ko_{X_L}(c-n,c-m),
\end{align*}
In particular $X_L$ is Fano if and only if $c < m$,
and rational for $c<n$.
\end{corollary}

\begin{proof}
Since $X_L^0=\emptyset$ 
the condition for $X_L$ to be Fano descends directly for the formula for the canonical bundle. The statement on rationality is Proposition \ref{razionalita}.
\end{proof}

The variety $Y_L$ is itself a linear section of a Segre variety, by
Proposition \ref{mi sa che bisogna dimostrarla}, as the folloing Lemma shows.

\begin{lemma}
The variety $Y_L$ is isomorphic to the complete intersection of $n$
hyperplanes in $\PP U \times \PP^L$ determined by $L \subset W$.
So the canonical bundle $\omega_{Y_L}$ equals $\ko_{Y_L}(n-m,n-c)$.
Moreover, for generic $L\subset W$, the determinantal variety $Z^L$ is smooth if and only if $c<2n-2m+5$.
\end{lemma}

\begin{proof}
The first statement follows from the very last item of Proposition
\ref{mi sa che bisogna dimostrarla}. Indeed, since $r=1$, $Y_L$ the
projectivization of the sheaf $\ke$, restricted to $\PP^L$.
Therefore, just as in the proof of Proposition \ref{mi sa che bisogna dimostrarla},
 $Y_L$ is the vanishing locus of the global section of $\ko_{
 \PP U\times \PP^L}(1,1)$ determined by the subspace $L^\perp \subset W^\vee$,
i.e. by $L \subset W$.

Note that $\ko_{Y_L}(0,1)\simeq \ko_{Y_L}(H)$ and $\ko_{Y_L}(1,0)=\ko_{Y_L}(Q)$.
The canonical bundle formula
follows by adjunction and agrees with Corollary \ref{cor:adjunction}.

The codimension in $\PP^L$ of the singular locus of $Z^L$ is
$2n-2m+4$ for a general choice of $L\subset W$. So $Z^L$ is smooth if
and only if $c < 2n-2m+5$, which gives the last statement.
\end{proof}

\begin{remark}
  Here, 
since $\ku\dual=T_{\PP U}(-1)$. By Proposition \ref{mi sa che bisogna dimostrarla}, the variety $Y_L$ can also be described as the projectivization of the cokernel sheaf of
  \begin{equation}
    \label{pino}
   L^{\perp} \otimes \ko_{\PP U} \to V\dual \otimes T_{\PP U}(-1), 
  \end{equation}
\end{remark}

The map appearing in (\ref{pino}) in the remark above, corresponds once again to the choice of $L \subset W$. 
Dually, for $Y_L$, Proposition \ref{razionalita} gives:

\begin{lemma}
The variety $Y_L$ is rational if $c>n$.
\end{lemma}  

\begin{proof}
  This is just Proposition \ref{razionalita}.
\end{proof}

Thanks to the constructions of section \ref{sect:CY}, we obtain the following corollary.

\begin{corollary}
If $c=n$, then $X_L$ and $Y_L$ are birational $(m-2)$-folds. If $m=n$ they are Calabi-Yau and have nef canonical divisor otherwise.
\end{corollary}

We resume the results of this section in Table \ref{table:r=1}.

\begin{table}[h!]
\small{
{
\begin{center}
\begin{tabular}{|l||l|l|l|l|}
\hline
& $c<m$ & $m \leq c < n$ & $c=n$ & $n < c$ \\
\hline
HPD Functor & \multicolumn{2}{c|}{$\Db(Y_L) \to \Db(X_L)$} & equivalence & $\Db(X_L) \to \Db(Y_L)$  \\
\hline
\hline

\multirow{2}{*}{$Y_L$} & \multirow{2}{*}{Fano visitor} & & \multirow{2}{*}{CY if $n=m$} & Rational \\
 & & & & Fano if $n=m$ \\
\hline
$X_L$ & Rational Fano & Rational &CY if $n=m$ & Fano visitor if $n=m$ \\
\hline

\end{tabular}
\end{center}
}
}
\caption{The Segre-determinantal duality.}
\label{table:r=1}
\end{table}

\section{Fano and rational varieties}\label{sect:fano}

\subsection{Representability into Fano varieties}
In this section, we consider question \ref{qu:bondal}. We start
by stating a straightforward consequence of Corollary \ref{cor:functors} and Lemma \ref{adjuncion}
(see also Table \ref{table:all-r}), which provides a large class of examples
of weakly Fano-visitor (see Def. \ref{whost}) varieties, up to categorical resolutions of singularities.

\begin{proposition}\label{Fano-visitor-allr}
Suppose that $n=m$.
If $c<rn$, then $Y_L^r$ and $(Z^L_r,\kr'_{\PP^L})$ are weakly Fano visitor. If $c>nr$, then $X_L^r$ 
and $(Z^r_L,\kr'_{\PP_L})$ are weakly Fano visitor.
\end{proposition}
If $r=1$ we have an interpretation of Proposition \ref{Fano-visitor-allr} for determinantal varieties.

\begin{corollary}\label{cor:determinantal-Fanovis}
Let $Z \subset \PP^k$ be a determinantal variety associated with a generic $m \times n$ matrix.
If $k<m-1$ then the categorical resolution of singularities of $Z$ is Fano visitor.
\end{corollary}

\begin{proof}
The determinantal variety $Z$ is $Z^L_{m-1}=
  \kz_{m,n}^{m-1}\cap \PP^L$ for a subspace $L\subset U\otimes V$ of codimension $k+1$. 
 Then we use results from Table \ref{table:r=1} and conclude.
\end{proof}

Corollary \ref{cor:determinantal-Fanovis} gives a positive answer to Question \ref{qu:bondal}
for almost every curve.

\begin{example}[Plane curves]\label{ex:plane-cves}
Let $C \subset \PP^2$ be a plane curve of degree $d \ge 4$. 
Then, it is well known (see \cite[\S 3]{bove:det}) that $C$ can be written as the determinant of a $d \times d$ matrix of linear forms.
In other words, we put $m=n=d$, $k=2$ and the inequality of Corollary \ref{cor:determinantal-Fanovis} is respected.
Hence any plane curve of degree at least four is a Fano-visitor, up to
resolution of singularities.\ 

On the other hand, one can check that the blow-up of $\PP^3$ along a plane cubic is Fano (see, {\it e.g.}, \cite[Proposition 3.1, (i)]{blanc-lamy}).
Hence any plane curve of positive genus is a Fano-visitor.
\end{example}

\begin{example}[More curves of general type]

Determinantal varieties with $n\neq m$ provide a wealth of examples of (even non plane) curves of general type that are Fano-visitor.

Let us make the case where  $\dim (Y_L^1) = \dim (Z^L)= 1$ explicit. We have $c=n-m+3$. From Table \ref{table:r=1} it is straightforward
to see that $Y_L^1$ is an elliptic curve (the Calabi-Yau case) if $m=n=c=3$; this yields indeed a plane cubic. On the other hand, we see that if $m=2$ then the
curve is rational for any value of $n$ since $c=n+1$, and if $m>3$ it is forced to be a curve of general type in $\PP^{c-1}$, which is Fano visitor if $c<m$.

The dual $X_L$ is a smooth variety of dimension $2m-5$. If $m=3$, we have that $Z_L$ is an elliptic curve. If $m > 3$,
we have $\dim Z_L \geq 3$. This gives quite a lot of examples of space curves of general type that are Fano visitors. 
Take for example $c=4,\ n=6$ and $m=5$. This gives a curve of genus 4 in $\PP^3$, complete intersection of two degree 5 determinantal hypersurfaces, whose derived category is fully faithfully embedded in the derived category of a rational Fano 5-fold in $\PP^{25}$.

\end{example}

\subsection{Rationality and categorical representability}
In this subsection, we consider Question \ref{qu:catrep=ration}.
The second consequence of Corollary \ref{cor:functors} is a large class of examples of rational varieties which are categorically
representable in codimension at least 2. For simplicity, let us assume that $r=1$, so that we already discussed in section \ref{sect:segre-det}
the rationality of the sections. We state the following Proposition in terms of Segre and determinantal varieties.

\begin{corollary}
The categorical resolution of a rational determinantal variety is categorically representable in codimension at least 2.
A rational linear section of the Segre variety $\PP^{n-1} \times \PP^{m-1} \subset \PP^{nm-1}$ is categorically 
representable in codimension at least 2.
\end{corollary}

\begin{proof}
First we observe that the Segre linear section $X_L$ is rational for $c<n$ and the determinantal linear section $Y_L$ for $c>n$ by Table \ref{table:r=1}. Then we recall from Corollary \ref{cor:functors} that, assuming $r=1$, it is exactly in these ranges  that we have the required functors and semiorthogonal decompositions. A computation of the dimensions of the linear sections, following the formulas in section \ref{proshe}, proves the claim.
\end{proof}

\subsection{Categorical resolution of the residual category of a determinantal Fano hypersurface} 
The Segre-determinantal HPD involves categorical resolutions for determinantal varieties, which is crepant if $n=m$.
In this subsection we consider the cases where such resolution gives a crepant categorical resolution
for nontrivial components of a semiorthogonal decomposition. For simplicity, we will consider only determinantal \it hypersurfaces\rm,
hence we need to assume $r=1$ and $m=n$. We will drop all the useless indexes.

Let $F$ be a smooth Fano variety such that $\Pic(F) = \ZZ[\ko_F(1)]$. The index of $F$ is the integer $i$ such that $\omega_F=\ko_F(-i)$. Kuznetsov observed
that this kind of varieties have a Lefschetz-type semiorthogonal decomposition.

\begin{lemma}\cite[Lemma 3.4]{kuz:fano}\label{uno}
Let $F$ be a smooth Fano variety of index $i$, then the collection $\ko_F(-i+1), \dots, \ko_F$ in $\Db(F)$ is exceptional.
\end{lemma}

\begin{corollary}\cite[Corollary 3.5]{kuz:fano}
For any smooth Fano variety $F$ of Picard rank $1$ and index $i$ we have the following semiorthogonal decomposition
\begin{equation}\label{fanodeco}
\Db(F)= \langle \ko_F(-i+1), \ldots, \ko_F,\cat{T}_F \rangle,
\end{equation}
where $\cat{T}_F= \{ E \in \Db(V)| H^\bullet(V,E(-k))=0\ for\ all\ 0\leq k \leq i-1 \}$.
\end{corollary}

The main technical tools used in the proof of Lemma \ref{uno}
are Kodaira vanishing Theorem and Serre duality. Before we proceed, we first need to broaden slightly the class of varieties for which
the semiorthogonal decomposition (\ref{fanodeco}) holds. 
In fact, we recall that Kodaira vanishing
holds also for varieties with rational singularities (for example, see \cite[I, Example 4.3.13]{lazarsfeld-book}), and the well-known fact that the canonical divisor of a Gorenstein variety is Cartier.

\begin{proposition}\label{propo-deco-for-singular}
Let $F$ be a projective Gorenstein variety with rational singularities.
Suppose that $\Pic(F) = \ZZ$, $\ko_F(1)$ is its (ample) generator and $K_F = \ko_F(-i)$, with $i>0$.
Then there is a semiorthogonal decomposition
$$\Db(F) = \langle \ko_F(-i+1), \ldots, \ko_F, \cat{T}_F \rangle.$$

This holds in particular if $F \subset \PP^k$ is an hypersurface of degree $d < k$
with rational singularities (in which case, $i=k-d$).
\end{proposition}

\begin{proof}
It is straightforward to check that the line bundle $\ko_F(i)$ is exceptional for any $i$.
To show the semiorthogonality, we use a vanishing theorem for varieties with rational singularities
(see \cite[I, Example 4.3.13]{lazarsfeld-book}), which states that
$$\mathrm{Ext}^{j}(\ko_F(s),\ko_F(t))\simeq \mathrm{Ext}^{j}(\ko_F,\ko_F(t-s))\simeq H^j(F,\ko_F(t-s))$$

vanishes for $j < \mathrm{dim}(F)$, and $s > t$.
Thanks to Serre duality
$$\mathrm{Ext}^{\mathrm{dim}(F)}(\ko_F(s),\ko_F(t))\simeq H^{\mathrm{dim}(F)}(F,\ko_F(t-s)) \simeq H^0(F,\ko_F(s+i-t))$$
and the latter group vanishes if $s+i-t < 0$, 
\end{proof}

Homological Projective Duality allows us to describe a resolution of singularities of $\cat{T}_F$ in the case where
$F$ is determinantal. This means that we consider $Z^L \subset \PP^L$ for some integers $m=n$ and for some linear subspace $L\subset U \otimes V$ of Fano type (that is, of degree $d<k+1$). The Springer resolution of $Z^L$ is then $Y_L$ and the
dual section of the Segre variety is $X_L$. Let us fix $L$, and drop it from the notations from now on. We want to describe
a categorical resolution of the category $\cat{T}_Z$ described in Proposition \ref{propo-deco-for-singular}.

We constructed a crepant categorical resolution of singularities $\Db(Z,\kr')$ of $Z$.
The category $\Db(Z,\kr')$ is equivalent to $\Db(Y)$, for $Y$ the corresponding fibre product of the linear section of the Springer resolution (see Theorem \ref{thm:main2}).
In particular, $Y$ is a (the fibre product over a) linear section of a projective bundle over $\PP^{d-1}$, since $d=n=m$ is the degree of $Z$.
Let us denote by $X$ the dual linear section of the Segre variety (notice in fact that $X$ is smooth).
Numerical computations provide a semiorthogonal decomposition
$$\Db(Z,\kr') \simeq \Db(Y) = \langle k-d+1 \text{ copies of } \Db(\PP^{d-1}), \Db(X) \rangle.$$

Hence $\Db(Z,\kr')$ is generated by by $d(k-d+1)$ exceptional objects and $\Db(X)$.

More precisely, the $j$-th occurrence of $\Db(\PP^{d-1})$ can be generated by the exceptional sequence $(\ko_{Y}(j,1), \ldots, \ko_{Y}(j,d))$, where we use 
the same notation $\ko_{Y}(a,b)$ as in Section \ref{sect:segre-det}, \it i.e. \rm $\ko_{Y_L}(0,1)\simeq \ko_{Y_L}(H)$ and $\ko_{Y_L}(1,0)=\ko_{Y_L}(Q)$.

This allows one to calculate a categorical resolution of singularities of $\cat{T}_{Z}$ which is
decomposed into $\Db(X)$ and exceptional objects.

\begin{proposition}
Let $Z$ be a Fano determinantal hypersurface of $\PP^k$, and $X$ the dual section of the Segre variety.
There is a strongly crepant categorical resolution 
$\widetilde{\cat{T}}_Z$ of $\cat{T}_{Z}$, admitting a semiorthogonal
decomposition by $\Db(X)$ and $(d-1)(k-d+1)$ exceptional objects.
\end{proposition}

\begin{proof}
Consider the resolution $p:Y \to Z$, and denote by $D$ its exceptional divisor.
We have proved that $\Db(Y) \simeq \Db(Z,\kr')$ is a categorical resolution of singularities
of $\Db(Z)$. In particular (see \cite{kuznet:singul}), this comes equipped with a functor $p^\vee : \mathrm{Perf}(Z) \to
\Db(X)$ admitting a right adjoint. Indeed, according to \cite{kuznet:singul}, to get such a pair for a variety $M$ with rational singularities,
one needs to consider a desingularization $q: N \to M$ with exceptional
divisor $E$, such that $\Db(E)$ admits a Lefschetz decomposition with respect to the conormal bundle.
In our case, we can just consider the Lefschetz decomposition with one component $\cat{B}_0 = \Db(D)$. Now we will check that all the hypotheses
of \cite[Theorem 1]{kuznet:singul} for the existence of such a categorical resolution are satisfied by the category generated by $\Db(X)$ and the exceptional objects.
So, in order to get a categorical resolution of singularities for $\cat{T}_{Z}$, let us consider the
functor $p^*$ introduced above and its action on the semiorthogonal
decomposition from Proposition \ref{propo-deco-for-singular}.

Let $\PP^L\simeq \PP^k$. There is a commutative diagram:

$$\xymatrix{
Y \ar[d]^p \ar[dr]^f \\
Z \ar[r]^g & \PP^k,
}$$

where the map $f$ is given by the restriction linear system $\vert \ko_{Y}(1,1) \vert $,
and the map $g$ is defined
by $\vert \ko_{Z} (1) \vert$. It follows that $p^*\ko_{Z}(k)
= \ko_{Y}(k,k)$, so that the exceptional sequence $\ko_{Z}(-k+d),\ldots,\ko_{Z}$ pulls back
to the exceptional sequence $\ko_{Y}(-k+d,-k+d),\ldots,\ko_{Y}$.

Now recall that $\ky^{1}_{d,d}$ is a projective bundle $s: \ky^{1}_{d,d} \simeq \PP(V\dual\otimes T_{\PP(U)}(-1)) \to \PP U$. The Lefschetz decomposition
of $\Db(\ky^{1}_{d,d})$
giving the HP-duality of Theorem \ref{thm:main1} is:
$$\Db(\ky^{1}_{d,d}) = \langle \cat{A}_{-j}\otimes \ko_{\PP(V\otimes \kq)}(-j), \ldots, \cat{A}_0 \rangle,$$
with $-j=1-d^2+d$, where $\cat{A}_0= \ldots = \cat{A}_j = s^* \Db(\PP U)$.
In particular, we can choose, for each occurrence of $s^* \Db(\PP U)$, an appropriate exceptional collection generating $\Db(\PP U)$
in order to get, after taking the linear sections (recall that $Y:=Y^1_L$, and $X:=X_L^1$):

$$\begin{array}{rl}
\Db(Y) =& \langle \ko_{Y}(-k+d,-k+d), \ldots, \ko_{Y}(-k+d,-k+2d-1), \\
& \ko_{Y}(-k+d+1,-k+d+1), \ldots, \ko_{Y}(-k+d+1,-k+2d), \\
& \ldots \\
& \ko_{Y}(0,0), \ldots, \ko_{Y}(0,d-1), \Db(X) \rangle.
\end{array}
$$

Now we can mutate all the exceptional objects which are not of the form 
$\ko_{Y}(-t,-t)$, for some $t$, to the right until we get 

$$\begin{array}{rl}
\Db(Y) =& \langle \ko_{Y}(-k+d,-k+d), \ldots, \ko_{Y}(-1,-1), \ko_{Y}, \\
& E_1, \ldots, E_{(d-1)(k-d+1)},  \Db(X) \rangle,
\end{array}
$$
where the $E_i$ are the exceptional objects resulting from the mutations. Hence, the first block is the pull-back from
$Z$ of the exceptional sequence $(\ko_{Z}(-k+d), \ldots, \ko_{Z})$, then by definition we get that the second block is the
categorical resolution of singularities for $\cat{T}_{Z}$.
\end{proof}

\begin{remark}
A particular and interesting case is given by determinantal cubics in $\PP^4$ and $\PP^5$. In both cases, the dual linear section
$X$ is empty. So, the numeric values give explicitly:
\begin{itemize}
\item If $Z$ is a determinantal cubic threefold, then the category $\cat{T}_{Z}$ admits a crepant categorical resolution
of singularities generated by 4 exceptional objects.
\item If $Z$ is a determinantal cubic fourfold, then the category $\cat{T}_{Z}$ admits a crepant categorical resolution
of singularities generated by 6 exceptional objects.
\end{itemize}
In the case of cubic threefolds and fourfolds with only one node, categorical resolution of singularities
of $\cat{T}_Z$ are described (see resp. \cite{bolognesi_bernardara:representability} and \cite{kuz:4fold}).
One should expect that these geometric descriptions carry over to the more degenerate case of determinantal cubics
- which are all singular. We haven't developed the (very long) calculations, but nevertheless we outline 
expectations about the geometrical nature of these categorical resolutions.

\medskip

In the 3-dimensional case, the 4 exceptional objects should correspond to a disjoint union of
two rational curves, arising as the geometrical resolution of singularities of the discriminant locus of a
projection $Z \to \PP^3$ from one of the six singular points. This discriminant locus is
composed by two twisted cubics intersecting in five points, and turns out to be
a degeneration of the $(3,2)$ complete intersection curve appearing in the one-node case (see \cite[Proposition 4.6]{bolognesi_bernardara:representability}).

\medskip

In the 4-dimensional case, the 6 exceptional objects should correspond to a disjoint union of
two Veronese-embedded planes (isomorphically projected to $\PP^4$), arising as the geometrical resolution of singularities of the discriminant locus of a
projection $Z \to \PP^4$ from one of the singular points. This discriminant locus is
composed by two cubic scrolls intersecting along a quintic elliptic curve, and turns out to be
a degeneration of the degree 6 K3 surface appearing in the one-node case (see \cite[\S 5]{kuz:4fold}).
\end{remark}

\end{document}